\documentclass{article}
\usepackage{amsfonts}
\usepackage{amsmath}
\usepackage{natbib}
\usepackage[colorlinks = true,            linkcolor = blue,            urlcolor  = blue,            citecolor = red,            anchorcolor = blue]{hyperref}
\usepackage{hyperref}

\newtheorem{theorem}{Theorem}

\newtheorem{corollary}[theorem]{Corollary}

\newtheorem{definition}[theorem]{Definition}

\newtheorem{lemma}[theorem]{Lemma}

\newenvironment{proof}[1][Proof]{\noindent\textbf{#1.} }{\ \rule{0.5em}{0.5em}}
\input{tcilatex}
\begin{document}

\title{Uniform Concentration and Symmetrization for Weak Interactions}
\author{Andreas Maurer\thanks{%
Adalbertstr. 55, 80799 Munich, Germany. Email: \emph{am@andreas-maurer.eu}}
\and Massimiliano Pontil\thanks{%
Computational Statistics and Machine Learning, Istituto Italiano di
Tecnologia, 16100 Genova, Italy.} \thanks{%
Department of Computer Science, University College London, WC1E 6BT London,
UK. Email: \emph{m.pontil@ucl.ac.uk}}}
\maketitle

\begin{abstract}
The method to derive uniform bounds with Gaussian and Rademacher
complexities is extended to the case where the sample average is replaced by
a nonlinear statistic. Tight bounds are obtained for U-statistics,
smoothened L-statistics and error functionals of l2-regularized algorithms.
\end{abstract}

\section{Introduction}

The purpose of this paper is to extend the method of Rademacher or Gaussian
complexities to a more general, nonlinear setting. Suppose that $\mathbf{X}%
=\left( X_{1},...,X_{n}\right) $ is a vector of independent random variables
with values in some space $\mathcal{X}$, $\mathbf{X}^{\prime }$ is iid to $%
\mathbf{X}$, and that $\mathcal{H}$ is a finite class of functions $h:%
\mathcal{X}\rightarrow \left[ 0,1\right] $. For $\mathbf{x}\in \mathcal{X}%
^{n}$ and $h\in \mathcal{H}$ we use $h\left( \mathbf{x}\right) $ to denote
the vector $h\left( \mathbf{x}\right) =\left( h\left( x_{1}\right)
,...,h\left( x_{n}\right) \right) \in \left[ 0,1\right] ^{n}$ and $\mathcal{H%
}\left( \mathbf{x}\right) =\left\{ h\left( \mathbf{x}\right) :h\in \mathcal{H%
}\right\} \subseteq \mathbb{R}^{n}$. Now let $f:\left[ 0,1\right]
^{n}\rightarrow \mathbb{R}$ be the sample average 
\begin{equation*}
f\left( s_{1},...,s_{n}\right) :=\frac{1}{n}\sum_{i=1}^{n}s_{i}\text{ for }%
s_{i}\in \mathbb{R}.
\end{equation*}%
Then it is not hard to show (see \cite{Bartlett 2002}, Theorem 8, or \cite%
{Ledoux 1991}, Lemma 6.3 and (4.8)) that%
\begin{equation}
\mathbb{E}\left[ \sup_{h\in \mathcal{H}}\mathbb{E}_{\mathbf{X}^{\prime }}%
\left[ f\left( h\left( \mathbf{X}^{\prime }\right) \right) \right] -f\left(
h\left( \mathbf{X}\right) \right) \right] \leq \frac{2}{n}\mathbb{E}\left[
R\left( \mathcal{H}\left( \mathbf{X}\right) \right) \right] \leq \frac{\sqrt{%
2\pi }}{n}~\mathbb{E}\left[ G\left( \mathcal{H}\left( \mathbf{X}\right)
\right) \right] ,  \label{symmetrization inequalities}
\end{equation}%
where the Rademacher and Gaussian averages of a subset $Y\subseteq \mathbb{R}%
^{n}$ are%
\begin{equation*}
R\left( Y\right) =\mathbb{E}\sup_{\mathbf{y}\in Y}\left\langle \mathbf{%
\epsilon ,y}\right\rangle \text{ and }G\left( Y\right) =\mathbb{E}\sup_{%
\mathbf{y}\in Y}\left\langle \mathbf{\gamma ,y}\right\rangle \text{.}
\end{equation*}%
Here $\mathbf{\epsilon }=\left( \epsilon _{1},...,\epsilon _{n}\right) $ and 
$\gamma =\left( \gamma _{1},...,\gamma _{n}\right) $ are vectors of
independent Rademacher and standard normal variables respectively.

The bounded difference inequality (Theorem \ref{Theorem bounded difference},
often called McDiarmid's inequality) shows that the random variable $%
\sup_{h\in \mathcal{H}}\mathbb{E}_{\mathbf{X}^{\prime }}\left[ f\left(
h\left( \mathbf{X}^{\prime }\right) \right) \right] -f\left( h\left( \mathbf{%
X}\right) \right) $ is sharply concentrated about its mean, and the
symmetrization inequalities (\ref{symmetrization inequalities}) lead to a
uniform bound on the estimation error (see \cite{Koltchinskii 2002} or \cite%
{Bartlett 2002}): for any $\delta \in \left( 0,1\right) $ with probability
at least $1-\delta $%
\begin{equation}
\sup_{h\in \mathcal{H}}\mathbb{E}_{\mathbf{X}^{\prime }}\left[ f\left(
h\left( \mathbf{X}^{\prime }\right) \right) \right] -f\left( h\left( \mathbf{%
X}\right) \right) \leq \frac{2}{n}\mathbb{E}\left[ R\left( \mathcal{H}\left( 
\mathbf{X}\right) \right) \right] +\sqrt{\frac{\ln \left( 1/\delta \right) }{%
2n}}.  \label{Uniform bound classical}
\end{equation}%
This fact has proven very useful in statistical learning theory, and many
techniques have been developed to bound Rademacher and Gaussian averages in
various contexts of classification, function learning, matrix completion,
multi-task learning and unsupervised learning (see e.g. \cite{Bartlett 2002}%
, \cite{Meir Zhang}, \cite{Ambroladze 05}, \cite{Kakade 2009}, \cite%
{KakadeEtAl 2012}, \cite{Biau Lugosi}).

The sample average is particularly simple and useful, but there are many
other interesting statistics, which are nonlinear, such as U-statistics,
quantiles, or M-estimators to estimate other distributional properties.
Concrete examples would be estimators of the median for economic
applications, or the Wilcoxon two-sample statistic, which plays a role in
the evaluation of ranking functions (\cite{Agarwal 2005}). Nonlinear
versions of (\ref{symmetrization inequalities}) and (\ref{Uniform bound
classical}) could be quite useful and make the abundance of techniques to
bound Rademacher and Gaussian averages available in a larger context.

Such an extension is possible, also for vector valued function classes, if
the statistic $f$ in question has the right kind of Lipschitz property and
is not too "far from linearity". To make this precise we make the following
definition.\bigskip 

\begin{definition}
Suppose $f:\mathcal{X}^{n}\rightarrow \mathbb{R}$. For $k\in \left\{
1,...,n\right\} $ and $y,y^{\prime }\in \mathcal{X}$, define the $k$-th
partial difference operator as 
\begin{equation*}
D_{yy^{\prime }}^{k}f\left( \mathbf{x}\right) =f\left(
...,x_{k-1},y,x_{k+1},...\right) -f\left( ...,x_{k-1},y^{\prime
},x_{k+1},...\right) \text{, for }\mathbf{x}\in \mathcal{X}^{n}.
\end{equation*}%
For $\mathcal{U}\subseteq \mathbb{R}^d$ we define seminorms $M_{Lip}$ and $%
J_{Lip}$ on the vector space of real functions $f:\mathcal{U}^{n}\rightarrow 
\mathbb{R}$ by%
\begin{eqnarray*}
M_{Lip}\left( f\right) &=&\max_{k}\sup_{\mathbf{x}\in \mathcal{U}^{n},y\neq
y^{\prime }\in \mathcal{U}}\frac{D_{yy^{\prime }}^{k}f\left( \mathbf{x}%
\right) }{\left\Vert y-y^{\prime }\right\Vert }\text{ and} \\
J_{Lip}\left( f\right) &=&n~\max_{k\neq l}\sup_{\mathbf{x}\in \mathcal{U}%
^{n},y\neq y^{\prime },z,z^{\prime }\in \mathcal{U}}\frac{D_{zz^{\prime
}}^{l}D_{yy^{\prime }}^{k}f\left( \mathbf{x}\right) }{\left\Vert y-y^{\prime
}\right\Vert }.
\end{eqnarray*}
\end{definition}

With these definitions we can extend the Gaussian part of the symmetrization
inequalities (\ref{symmetrization inequalities}) to nonlinear
statistics.\bigskip

\begin{theorem}
\label{Theorem Main} Let $\mathbf{X}=\left( X_{1},...,X_{n}\right) $ be a
vector of independent random variables with values in $\mathcal{X}$, $%
\mathbf{X}^{\prime }$ iid to $\mathbf{X}$, let $\mathcal{U}\subseteq \mathbb{%
R}^{d}$, let $\mathcal{H}$ be a finite class of functions $h:\mathcal{%
X\rightarrow U}$ and let $\mathcal{H}\left( \mathbf{X}\right) =\left\{
h\left( \mathbf{x}\right) :h\in \mathcal{H}\right\} \subseteq \mathbb{R}%
^{dn} $. Then for $f:\mathcal{U}^{n}\rightarrow \mathbb{R}$%
\begin{equation}
\mathbb{E}\left[ \sup_{h\in \mathcal{H}}\mathbb{E}_{\mathbf{X}^{\prime }}%
\left[ f\left( h\left( \mathbf{X}^{\prime }\right) \right) \right] -f\left(
h\left( \mathbf{X}\right) \right) \right] \leq \sqrt{2\pi }\left(
2M_{Lip}\left( f\right) +J_{Lip}\left( f\right) \right) ~\mathbb{E}\left[
G\left( \mathcal{H}\left( \mathbf{X}\right) \right) \right] .
\label{Main inequality}
\end{equation}%
\bigskip
\end{theorem}

Remarks:

\begin{enumerate}
\item If $d=1$ and $f$ is the arithmetic mean, then it is easy to see that $%
M_{Lip}\left( f\right) =1/n$ and $J_{Lip}\left( f\right) =0$, so the
Gaussian version of (\ref{symmetrization inequalities}) is recovered up to a
constant factor of $2$.

\item Since the right hand side of (\ref{Main inequality}) is invariant
under a sign-change of $f$, the same bounds hold for $\sup_{h\in \mathcal{H}%
}f\left( h\left( \mathbf{X}\right) \right) -\mathbb{E}\left[ f\left( h\left( 
\mathbf{X}^{\prime }\right) \right) \right] $.

\item In many applications the Gaussian average $G\left( \mathcal{H}\left( 
\mathbf{X}\right) \right) $ can be bounded in the same way as the Rademacher
average. In general $G\left( \mathcal{H}\left( \mathbf{X}\right) \right) $
can be bounded by $R\left( \mathcal{H}\left( \mathbf{X}\right) \right) $
with an additional factor of $3\sqrt{\ln \left( n+1\right) }$ (see \cite%
{Ledoux 1991}, (4.9)).

\item Finite cardinality of $\mathcal{H}$ is required to avoid problems of
measurability and should not be too disturbing, because the cardinality of $%
\mathcal{H}$ can be arbitrarily large. For infinite $\mathcal{H}$ one can
replace expressions like $\mathbb{E}\left[ \sup_{h\in \mathcal{H}}\left(
.\right) \right] $ by $\sup_{\mathcal{H}_{0}\subset \mathcal{H},\left\vert 
\mathcal{H}_{0}\right\vert <\infty }\mathbb{E}\left[ \sup_{h\in \mathcal{H}%
_{0}}\left( .\right) \right] $.\bigskip
\end{enumerate}

For a given statistic $f$ the key to the application of Theorem \ref{Theorem
Main} is the verification that $M_{Lip}\left( f\right) $ and $J_{Lip}\left(
f\right) $ are of order $O\left( 1/n\right) $. This is true for the sample
average, but also for

\begin{itemize}
\item U- and V-statistics of all orders with coordinate-wise Lipschitz
kernels. This includes multi-sample cases, such as smoothened versions of
the Wilcoxon two-sample-statistic. A corresponding application to ranking is
sketched in Section \ref{Subsetction Ranking}.

\item Lipschitz L-statistics. These are weighted averages of order
statistics with Lipschitz weighting functions and include smoothened
approximations to medians, or smoothened estimators for quantiles. In
Section \ref{Subsection L-statistics} a potential application to robust
clustering is discussed.

\item a class of M-estimators with strongly convex objectives, in particular
error functionals of $\ell _{2}$-regularized classification or function
estimation. In Section \ref{Subsection L_2 regularization} we sketch an
application to representation learning.
\end{itemize}

This list is not exhaustive and other examples can be generated using the
fact that $M_{Lip}$ and $J_{Lip}$ are seminorms. Also, if $\mathcal{U}%
\subseteq \mathbb{R}^d$ is bounded and $M_{Lip}\left( f\right) $ and $%
J_{Lip}\left( f\right) $ are of order $O\left( 1/n\right) $, then every
twice differentiable function with bounded derivatives when composed with $f$
has the same property (see \cite{Maurer 2018}).\bigskip

The seminorms $M_{Lip}$ and $J_{Lip}$ are strongly related to the seminorms $%
M$ and $J$ introduced in \cite{Maurer 2018}. For $f:\mathcal{X}%
^{n}\rightarrow \mathbb{R}$, where $\mathcal{X}$ can be any set, they are
defined as%
\begin{eqnarray*}
M\left( f\right) &=&\max_{k}\sup_{\mathbf{x}\in \mathcal{X}^{n},y,y^{\prime
}\in \mathcal{X}}D_{y,y^{\prime }}^{k}f\left( \mathbf{x}\right) \text{ and}
\\
J\left( f\right) &=&n\text{ }\max_{k,l:k\neq l}\sup_{\mathbf{x}\in \mathcal{X%
}^{n},y,y^{\prime },z,z^{\prime }\in \mathcal{X}}D_{z,z^{\prime
}}^{l}D_{y,y^{\prime }}^{k}f\left( \mathbf{x}\right) .
\end{eqnarray*}%
$M$ and $J$ control the nonlinear generalizations of several properties of
linear statistics, such as Bernstein's inequality, sample-efficient variance
estimation, empirical Bernstein bounds and Berry-Esseen type bounds of
normal approximation (see \cite{Maurer 2017} and \cite{Maurer 2018}). If $%
\mathcal{U}$ is bounded with diameter $\Delta $, then clearly $M\left(
f\right) \leq M_{Lip}\left( f\right) \Delta $ and $J\left( f\right) \leq
J_{Lip}\left( f\right) \Delta $, and the results in \cite{Maurer 2018} can
be reformulated in terms of $M_{Lip}$ and $J_{Lip}$. In particular, if $%
\mathcal{U}$ is bounded and $M_{Lip}\left( f\right) $ and $J_{Lip}\left(
f\right) $ are of order $O\left( 1/n\right) $, then $f\,$is a weakly
interactive function as defined in \cite{Maurer 2018}.\bigskip

Theorem \ref{Theorem Main}, the definition of $M\left( f\right) $ and the
bounded difference inequality (Theorem \ref{Theorem bounded difference})
applied to the random variable $\sup_{h\in \mathcal{H}}\mathbb{E}\left[
f\left( h\left( \mathbf{X}^{\prime }\right) \right) \right] -f\left( h\left( 
\mathbf{X}\right) \right) $ yield the nonlinear extension of (\ref{Uniform
bound classical}).

\begin{corollary}
\label{Corollary uniform concentration}Under the conditions of Theorem \ref%
{Theorem Main}, for any $\delta \in \left( 0,1\right) $, with probability at
least $1-\delta ,$%
\begin{multline*}
\sup_{h\in \mathcal{H}}\mathbb{E}\left[ f\left( h\left( \mathbf{X}^{\prime
}\right) \right) \right] -f\left( h\left( \mathbf{X}\right) \right) \\
\leq \sqrt{2\pi }\left( 2M_{Lip}\left( f\right) +J_{Lip}\left( f\right)
\right) ~\mathbb{E}\left[ G\left( \mathcal{H}\left( \mathbf{X}^{\prime
}\right) \right) \right] +M\left( f\right) \sqrt{n\ln \left( 1/\delta
\right) }.
\end{multline*}
\end{corollary}

The next section is devoted to applications, then we prove Theorem \ref%
{Theorem Main}. An appendix contains some technical material.

\section{Applications}

In the sequel we sketch some potential applications and exhibit some generic
classes of statistics, to which Theorem \ref{Theorem Main} and Corollary \ref%
{Corollary uniform concentration} can be applied.

\subsection{Ranking, U- and V-statistics\label{Subsetction Ranking}}

An example for the application of Theorem \ref{Theorem Main} is given by the
following variant of the Wilcoxon-two-sample statistic, which we simplify
for the purpose of illustration. Let $n$ be an even integer, $\ell :\mathbb{R%
}\rightarrow \left[ 0,1\right] $ and define $\hat{A}_{\ell }:\mathbb{R}%
^{n}\rightarrow \mathbb{R}$ by%
\begin{equation*}
\hat{A}_{\ell }\left( x_{1},...,x_{n}\right) =\frac{4}{n^{2}}%
\sum_{i=1}^{n/2}\sum_{j=n/2+1}^{n}\ell \left( x_{i}-x_{j}\right) .
\end{equation*}%
Now suppose that $\mu _{+}$ and $\mu _{-}$ are two probability measures on
some space $\mathcal{X}$, and we construct a sample $\mathbf{X}=\left(
X_{1},...,X_{n}\right) $ by drawing the first half of $\mathbf{X}$ iid from $%
\mu _{+}$ and the second half iid from $\mu _{-}$, that is $\mathbf{X}\sim
\mu _{+}^{n/2}\times \mu _{-}^{n/2}$. Now let $h:\mathcal{X}\rightarrow 
\mathbb{R}$ be some function. If $\ell =1_{\left( 0,\infty \right) }$ is the
indicator of the positive reals, then $\hat{A}_{\ell }\left( h\left( \mathbf{%
X}\right) \right) $ is evidently an unbiased estimator for 
\begin{equation}
\Pr_{\left( x,y\right) \sim \mu _{+}\times \mu _{-}}\left\{ h\left( x\right)
>h\left( y\right) \right\} ,  \label{AUC of h}
\end{equation}%
the "area under the ROC Curve" (AUC) (as explained in \cite{Agarwal 2005}),
and provides a criterion for the evaluation of $h$ as a ranking functions.
In this case $\hat{A}_{\ell }$ is the proper Wilcoxon statistic (apart from
the fact that we didn't worry about ties and consider a balanced sample for
simplicity), but other loss functions $\ell $ come into play if a good
ranking function is to be chosen from a set of candidates (see \cite{Ying}).

Let us assume that $\ell $ has Lipschitz constant $L$. Applying the partial
difference operator to the function $\hat{A}_{\ell }$, at first for $k\leq
n/2$, we find for any $y,y^{\prime }\in \mathbb{R}$%
\begin{equation*}
D_{y,y^{\prime }}^{k}\hat{A}_{\ell }\left( \mathbf{x}\right) =\frac{4}{n^{2}}%
\sum_{j=n/2+1}^{n}\ell \left( y-x_{j}\right) -\ell \left( y^{\prime
}-x_{j}\right) \leq \frac{2L}{n}\left\vert y-y^{\prime }\right\vert .
\end{equation*}%
Together with the analogous argument for $k>n/2$ this gives the bound%
\begin{equation}
M_{Lip}\left( \hat{A}_{\ell }\right) =\max_{k}\sup_{\mathbf{x}\in \mathbb{R}%
^{n},y\neq y^{\prime }\in \mathbb{R}}\frac{D_{y,y^{\prime }}^{k}\hat{A}%
_{\ell }\left( \mathbf{x}\right) }{\left\vert y-y^{\prime }\right\vert }\leq 
\frac{2L}{n}.  \label{First order Lipschitz response AUC}
\end{equation}%
In the same way one shows that $M\left( \hat{A}_{\ell }\right) \leq 2/n$. To
bound $J_{Lip}\left( \hat{A}_{\ell }\right) $ first let $k\leq n/2$, $l\neq
k $ and $y,y^{\prime },z,z^{\prime }\in \mathbb{R}$. Then%
\begin{align*}
& \left. D_{z,z^{\prime }}^{l}D_{y,y^{\prime }}^{k}\hat{A}_{\ell }\left( 
\mathbf{x}\right) =\frac{4}{n^{2}}\sum_{j=n/2+1}^{n}D_{z,z^{\prime
}}^{l}\left( \ell \left( y-x_{j}\right) -\ell \left( y^{\prime
}-x_{j}\right) \right) \right. \\
& =\left\{ 
\begin{array}{ccc}
0 & \text{if} & l\leq n/2 \\ 
\leq \frac{4}{n^{2}}\left( \ell \left( y-z\right) -\ell \left( y^{\prime
}-z\right) -\ell \left( y-z^{\prime }\right) +\ell \left( y^{\prime
}-z^{\prime }\right) \right) & \text{if} & l>n/2%
\end{array}%
\right. \leq \frac{8L}{n^{2}}\left\vert y-y^{\prime }\right\vert \text{,}
\end{align*}%
and analogous reasoning for $k>n/2$ gives%
\begin{equation}
J_{Lip}\left( \hat{A}_{\ell }\right) =n\max_{k,l:k\neq l}\sup_{\mathbf{x}\in 
\mathbb{R}^{n},y,y^{\prime },z,z^{\prime }\in \mathbb{R}}\frac{%
D_{z,z^{\prime }}^{l}D_{y,y^{\prime }}^{k}\hat{A}_{\ell }\left( \mathbf{x}%
\right) }{\left\vert y-y^{\prime }\right\vert }\leq \frac{8L}{n}.
\label{Second order Lipschitz response AUC}
\end{equation}%
Now suppose that $\mathcal{H}$ is a set of candidate ranking functions $h:%
\mathcal{X}\rightarrow \mathbb{R}$, for example a ball of linear functionals
in a RKHS. We wish to choose $h\in \mathcal{H}$ so as to maximize (\ref{AUC
of h}). If we choose $\ell \leq 1_{\left( 0,\infty \right) }$, then
Corollary \ref{Corollary uniform concentration} states that for every $%
\delta \in \left( 0,1\right) $ with probability at least $1-\delta $ in $%
\mathbf{X}$ we have for every potential ranking function $h\in \mathcal{H\,\ 
}$that%
\begin{eqnarray*}
\Pr_{\left( X,Y\right) \sim \mu _{+}\times \mu _{-}}\left\{ h\left( X\right)
>h\left( Y\right) \right\} &=&\mathbb{E}\left[ \hat{A}_{1_{\left( 0,\infty
\right) }}\left( h\left( \mathbf{X}^{\prime }\right) \right) \right] \geq 
\mathbb{E}\left[ \hat{A}_{\ell }\left( h\left( \mathbf{X}^{\prime }\right)
\right) \right] \\
&\geq &\hat{A}_{\ell }\left( h\left( \mathbf{X}\right) \right) -\frac{12%
\sqrt{2\pi }L~\mathbb{E}\left[ G\left( \mathcal{H}\left( \mathbf{X}^{\prime
}\right) \right) \right] }{n}-2\sqrt{\frac{\ln \left( 1/\delta \right) }{n}},
\end{eqnarray*}%
so as to justify the strategy to optimize the AUC by the maximization of the
empirical surrogate $\hat{A}_{\ell }\left( h\left( \mathbf{X}\right) \right) 
$. Similar bounds are obtained in \cite{Clemencon 2008}, even with fast
rates under some additional assumptions. The point here is to illustrate the
simplicity of only needing to verify the first- and second-order response
properties (\ref{First order Lipschitz response AUC}) and (\ref{Second order
Lipschitz response AUC}).\bigskip

A generalization of this example concerns the generic classes of V- and
U-statistics. Let $\mathcal{U}\subseteq \mathbb{R}^{d}$, $m\leq n$ and for
each $\mathbf{j}\in \left\{ 1,...,n\right\} ^{m}$ let $\kappa _{\mathbf{j}}:%
\mathcal{U}^{m}\rightarrow \mathbb{R}$. Define $V$,$U:\mathcal{U}%
^{n}\rightarrow \mathbb{R}$ by%
\begin{eqnarray*}
V\left( \mathbf{x}\right) &=&n^{-m}\sum_{\mathbf{j}\in \left\{
1,...,n\right\} ^{m}}\kappa _{\mathbf{j}}\left(
x_{j_{1}},...,x_{j_{m}}\right) \\
U\left( \mathbf{x}\right) &=&\binom{n}{m}^{-1}\sum_{1\leq
j_{1}<...<j_{m}\leq n}\kappa _{\mathbf{j}}\left(
x_{j_{1}},...,x_{j_{m}}\right) .
\end{eqnarray*}%
The next theorem shows that $V$ and $U$ inherit the seminorm properties of
the worst kernel $\kappa _{\mathbf{j}}$, scaled down by a factor of $m/n$
and $m^{2}/n$ respectively.

\begin{theorem}
\label{Theorem U-statistic}Let $f$ be either $V$ or $U$ and $\mathcal{U}%
\subseteq \mathbb{R}^{d}$. Suppose that for all multi-indices $\mathbf{j}$
we have $M_{Lip}\left( \kappa _{\mathbf{j}}\right) \leq L$. Then $%
M_{Lip}\left( f\right) \leq Lm/n$ and $J_{Lip}\left( f\right) \leq Lm^{2}/n$%
. If $M\left( \kappa _{\mathbf{j}}\right) \leq B$ for all $\mathbf{j}$ then $%
M\left( f\right) \leq Bm/n$ and $J\left( f\right) \leq Bm^{2}/n$.
\end{theorem}

The easy proof is given in Appendix \ref{Appendix Proof Ustat}.
Symmetrization inequalities and uniform bounds are then immediate from
Theorem \ref{Theorem Main} and Corollary \ref{Corollary uniform
concentration}, without any symmetry assumptions on kernels or variables.

\subsection{Lipschitz L-statistics and robust clustering\label{Subsection
L-statistics}}

Let $\mathcal{U}\subseteq \mathbb{R}$ be a bounded interval of diameter $%
\Delta $ and use $\left( x_{\left( 1\right) },...,x_{\left( n\right)
}\right) $ to denote the order statistic of $\mathbf{x}\in \mathcal{U}^{n}$.
Let $F:\left[ 0,1\right] \rightarrow \mathbb{R}$ have supremum norm $%
\left\Vert F\right\Vert _{\infty }$ and Lipschitz-constant $\left\Vert
F\right\Vert _{Lip}$ and consider the function%
\begin{equation}
\mathcal{L}_{F}\left( \mathbf{x}\right) =\frac{1}{n}\sum_{i=1}^{n}F\left(
i/n\right) x_{\left( i\right) }\text{.}  \label{L-statistic}
\end{equation}%
The following result is shown in \cite{Maurer 2018}.

\begin{theorem}
\label{Theorem LstatWeak}For $\alpha ,\beta \in \mathbb{R}$ let $\left[ %
\left[ \alpha ,\beta \right] \right] $ denote the interval $\left[ \min
\left\{ \alpha ,\beta \right\} ,\max \left\{ \alpha ,\beta \right\} \right] $%
. Then%
\begin{eqnarray}
\left\vert D_{y,y^{\prime }}^{k}\mathcal{L}_{F}\left( \mathbf{x}\right)
\right\vert &\leq &\frac{\left\Vert F\right\Vert _{\infty }\text{diam}\left( %
\left[ \left[ y,y^{\prime }\right] \right] \right) }{n}
\label{Lstat1 conditio} \\
\left\vert D_{z,z^{\prime }}^{l}D_{y,y^{\prime }}^{k}\mathcal{L}_{F}\left( 
\mathbf{x}\right) \right\vert &\leq &\frac{\left\Vert F\right\Vert _{Lip}%
\text{diam}\left( \left[ \left[ z,z^{\prime }\right] \right] \cap \left[ %
\left[ y,y^{\prime }\right] \right] \right) }{n^{2}}  \label{Lstat2 conditio}
\end{eqnarray}%
for any $\mathbf{x}\in \left[ 0,1\right] ^{n},$\textbf{\ }all $k\neq l$ and
all $y,y^{\prime },z,z^{\prime }\in \left[ 0,1\right] $.
\end{theorem}

It follows that $M\left( \mathcal{L}_{F}\right) \leq \Delta \left\Vert
F\right\Vert _{\infty }/n,M_{Lip}\left( \mathcal{L}_{F}\right) \leq
\left\Vert F\right\Vert _{\infty }/n$ and $J_{Lip}\left( \mathcal{L}%
_{F}\right) \leq \Delta \left\Vert F\right\Vert _{Lip}/n$. For a $\mathcal{U}
$-valued function class $\mathcal{H}$ Corollary \ref{Corollary uniform
concentration} implies the following uniform bound. For every $\delta \in
\left( 0,1\right) $ with probability at least $1-\delta $ in $\mathbf{X}$
that%
\begin{multline*}
\left\vert \sup_{h\in \mathcal{H}}\mathbb{E}\left[ \mathcal{L}_{F}\left(
h\left( \mathbf{X}^{\prime }\right) \right) \right] -\mathcal{L}_{F}\left(
h\left( \mathbf{X}\right) \right) \right\vert \\
\leq \frac{\sqrt{2\pi }\left( \Delta \left\Vert F\right\Vert
_{Lip}+2\left\Vert F\right\Vert _{\infty }\right) ~\mathbb{E}\left[ G\left( 
\mathcal{H}\left( \mathbf{X}^{\prime }\right) \right) \right] }{n}+\Delta
\left\Vert F\right\Vert _{\infty }\sqrt{\frac{\ln \left( 2/\delta \right) }{n%
}}.
\end{multline*}

Lipschitz L-statistics generalize the arithmetic mean, which is obtained by
choosing $F$ identically $1$. Other choices of $F$ lead to smoothely trimmed
means or smoothened sample-quantiles.

A potential use is in robust learning. It often happens that an objective
can be minimized very well only if a small proportion of outliers is trimmed
away previously. The problem is that minimization must already be performed
to identify the outliers, which suggests a procedure to re-sort the sample
according to current losses previous to each optimization step which then
disregards an upper percentile of losses. Since this generally results in
non-convex algorithms, it seems natural to consider problems which are
already non-convex to begin with.

We illustrate this idea in the case of $K$-means clustering (see \cite%
{Garcia 2007}). Here we seek a collection $\mathbf{c}=\left(
c_{1},...,c_{K}\right) $ of vectors in some ball $\mathbb{B\subseteq }$ $%
\mathbb{R}^{m}$ such that for a given random vector $X$ distributed in $%
\mathbb{B}$ the quantity $\mathbb{E}\left[ \ell \left( \mathbf{c}%
,X_{i}\right) \right] $ is small, where $\ell \left( \mathbf{c},X\right)
=\min_{k\in \left\{ 1,...,K\right\} }\left\Vert X-c_{k}\right\Vert ^{2}$.
For a sample $\mathbf{X}=\left( X_{1},...,X_{n}\right) $ the standard
strategy tries to find $\mathbf{c}\in \mathbb{B}^{K}$ so as to minimize the
arithmetic mean of the vector $\left( \ell \left( \mathbf{c},X_{1}\right)
,...,\ell \left( \mathbf{c},X_{n}\right) \right) $. Uniform bounds on the
estimation error have been given in \cite{Biau Lugosi}.

Now we assume that a significant portion of the data (say 25\%) consists of
noise, which is likely to affect the positions of the centers, but we are
happy to cluster only the remaining 75\%, which we expect to cluster well.
For $\zeta \in \left[ 0,1/4\right] $ let $F_{\zeta }:\left[ 0,1\right]
\rightarrow \mathbb{R}$ be the function%
\begin{equation*}
F_{\zeta }\left( t\right) =\left\{ 
\begin{array}{ccc}
4/3 & \text{if} & t\in \left[ 0,3/4-\zeta \right] \\ 
-\frac{2}{3\zeta }\left( t-3/4-\zeta \right) & \text{if} & t\in (3/4-\zeta
,3/4+\zeta ] \\ 
0 & \text{if} & t\in (3/4+\zeta ,0]%
\end{array}%
\right. .
\end{equation*}%
Then $F_{0}$ is the step function which drops from $4/3$ to zero at $t=3/4$
and $\mathcal{L}_{F_{0}}$ is a sample quantile, averaging the lower 75\%. If 
$\zeta \in (0,1/4]$ then $F_{\zeta }$ is an approximation to $F_{0}$ with
Lipschitz constant $2/\left( 3\zeta \right) $ and $\mathcal{L}_{F_{\zeta }}$
is an approximation to the sample quantile. Consider the algorithm%
\begin{equation*}
\min_{\mathbf{c\in }\mathbb{B}^{K}}\mathcal{L}_{F_{\zeta }}\left( \ell
\left( \mathbf{c},X_{1}\right) ,...,\ell \left( \mathbf{c},X_{n}\right)
\right) \text{.}
\end{equation*}%
The uniform bound above then provides a statistical performance guarantee
for this algorithm with respect to the transductive objective $\mathbb{E}%
\left[ \mathcal{L}_{F_{\zeta }}\left( \ell \left( \mathbf{c},X_{1}\right)
,...,\ell \left( \mathbf{c},X_{n}\right) \right) \right] $ (for a bound on
the Gaussian average of $\left\{ \left( \ell \left( \mathbf{c},X_{1}\right)
,...,\ell \left( \mathbf{c},X_{n}\right) \right) :\mathbf{c\in }\mathbb{B}%
^{K}\right\} $ see \cite{Biau Lugosi}). This method is a smoothened version
of the trimmed-$K$-means algorithms as described in \cite{Cuesta 1997} .

The idea of replacing the arithmetic mean of the objective function by a
smoothened sample-quantile can be applied to other methods of supervised or
unsupervised learning. For example the uniform bound would apply to support
vector machines, but replacing a convex problem by a non-convex one seems
less attractive.

\subsection{Differentiation, $\ell _{2}$-regularization and representation
learning\label{Subsection L_2 regularization}}

For smooth statistics the seminorms $M$, $M_{Lip}$ and $J_{Lip}$ can often
be bounded by differentiation. If $\mathcal{U}\subseteq \mathbb{R}^{d}$ is
open and $f:\mathcal{U}^{n}\rightarrow \mathbb{R}$ is $C^{2}$ then for $%
k,l\in \left\{ 1,...,n\right\} $ and $i,j\in \left\{ 1,...,d\right\} $ the
function $\left( \partial f/\partial x_{ki}\right) \left( \mathbf{x}\right) $
is simply the partial derivative of $f$ in the $\left( k,i\right) $%
-coordinate. Likewise $\left( \partial ^{2}f/\partial x_{ki}\partial
x_{lj}\right) \left( \mathbf{x}\right) $ is the partial derivative
corresponding to the coordinate pair $\left( \left( k,i\right) ,\left(
l,j\right) \right) $. We now introduce the notation $\partial _{k}f$ for the
vector valued function $\partial _{k}f:\mathcal{U}^{n}\rightarrow \mathbb{R}%
^{d}$ 
\begin{equation*}
\partial _{k}f\left( \mathbf{x}\right) =\left( \frac{\partial f}{\partial
x_{k1}}\left( \mathbf{x}\right) ,...,\frac{\partial f}{\partial x_{kd}}%
\left( \mathbf{x}\right) \right)
\end{equation*}%
and $\partial _{kl}f$ for the matrix valued function $\partial _{kl}f:%
\mathcal{U}^{n}\rightarrow \mathbb{R}^{d\times d}$%
\begin{equation*}
\partial _{kl}f=\left( 
\begin{array}{ccc}
\frac{\partial ^{2}f}{\partial x_{k1}\partial x_{l1}}\left( \mathbf{x}\right)
& ... & \frac{\partial ^{2}f}{\partial x_{kd}\partial x_{l1}}\left( \mathbf{x%
}\right) \\ 
... & ... & ... \\ 
\frac{\partial ^{2}f}{\partial x_{k1}\partial x_{ld}}\left( \mathbf{x}\right)
& ... & \frac{\partial ^{2}f}{\partial x_{kd}\partial x_{ld}}\left( \mathbf{x%
}\right)%
\end{array}%
\right) .
\end{equation*}%
With $\left\Vert \partial _{k}f\right\Vert =\sup_{\mathbf{x}\in \mathcal{U}%
^{n}}\left\Vert \partial _{k}f\left( \mathbf{x}\right) \right\Vert $ we
denote the supremum of the euclidean norm $\left\Vert \partial _{k}f\left( 
\mathbf{x}\right) \right\Vert $ of the vector $\partial _{k}f\left( \mathbf{x%
}\right) $ in $\mathcal{U}^{n}$, and with $\left\Vert \partial
_{kl}f\right\Vert =\sup_{\mathbf{x}\in \mathcal{U}^{n}}\left\Vert \partial
_{kl}f\left( \mathbf{x}\right) \right\Vert $ the supremum of the operator
norm $\left\Vert \partial _{kl}f\left( \mathbf{x}\right) \right\Vert _{op}$
of the matrix $\partial _{k}f\left( \mathbf{x}\right) $ in $\mathcal{U}^{n}$%
.\bigskip

\begin{theorem}
\label{Theorem Differentiation Seminorm bound}If $\mathcal{U}\subseteq 
\mathbb{R}^{d}$ is convex and bounded with diameter $\Delta $ and $f:%
\mathcal{U}^{n}\rightarrow \mathbb{R}$ extends to a $C^{2}$-function on an
open set $\mathcal{V}$ containing $\mathcal{U}^{n}$ then $M_{Lip}\left(
f\right) \leq \max_{k}\left\Vert \partial _{k}f\right\Vert $ and $%
J_{Lip}\left( f\right) \leq n\Delta \max_{k\neq l}\left\Vert \partial
_{kl}f\right\Vert $.\bigskip
\end{theorem}

This is proved in Appendix \ref{Appendix proof differentiation}. The uniform
estimation properties of a smooth statistic can therefore be described in
terms of bounds on the partial derivatives. Good results are obtained if
first order partial derivatives are of order $O\left( 1/n\right) $ and
second order derivatives are of order $O\left( 1/n^{2}\right) $.

We sketch an application to representation learning. Let $\mathbb{B}$ be the
unit ball $\mathbb{R}^{d}$ and let $\mathcal{U}=\mathbb{B\times }\left[ -1,1%
\right] $. Fix $\lambda \in \left( 0,1\right) $. For $\mathbf{x=}\left(
\left( z_{1},y_{1}\right) ,...,\left( z_{n},y_{n}\right) \right) \in 
\mathcal{U}^{n}$ regularized least squares returns the vector%
\begin{equation*}
w\left( \mathbf{x}\right) =\arg \min_{w\in \mathbb{R}^{d}}\frac{1}{n}%
\sum_{i=1}^{n}\left( \left\langle w,z_{i}\right\rangle -y_{i}\right)
^{2}+\lambda \left\Vert w\right\Vert ^{2}\text{.}
\end{equation*}%
The "empirical error" $f$ on $\mathcal{U}^{n}$ is then%
\begin{equation*}
f\left( \mathbf{x}\right) =\frac{1}{n}\sum_{i=1}^{n}\left( \left\langle
w\left( \mathbf{x}\right) ,z_{i}\right\rangle -y_{i}\right) ^{2}.
\end{equation*}%
Using the well known explicit formula for $w\left( \mathbf{x}\right) $ and $%
f\left( \mathbf{x}\right) $ one can show (see \cite{Maurer 2017}) by
differentiation that there are absolute constants $c_{1}$ and $c_{2}$, such
that for any $k,l\in \left\{ 1,...,n\right\} $, $k\neq l$,%
\begin{equation}
\left\Vert \partial _{k}f\right\Vert \leq \frac{c_{1}\lambda ^{-2}}{n}\text{
and }\left\Vert \partial _{kl}f\right\Vert \leq \frac{c_{2}\lambda ^{-3}}{%
n^{2}}\text{.}  \label{Derivativebound least squares}
\end{equation}%
so, taking the diameter of $\mathcal{U}$ into account, we have $M\left(
f\right) \leq c_{3}n^{-1}\lambda ^{-2}$, $M_{Lip}\left( f\right) \leq
c_{4}n^{-1}\lambda ^{-2}$ and $J_{Lip}\left( f\right) \leq
c_{4}n^{-1}\lambda ^{-3}$.

Now let $\mathcal{H}$ be a class of representations of some underlying space 
$\mathcal{X}$ of labeled data, that is functions $h:\mathcal{X\rightarrow U}$%
, which leave the labels invariant, and we wish to find an optimal
representation. If we plan to use ridge regression in the top layer, the
obvious criterion for the quality of the representation on a sample $\mathbf{%
X}\in \mathcal{X}^{n}$\textbf{\ }is 
\begin{equation*}
\mathbb{E}\left[ f\left( h\left( \mathbf{X}\right) \right) \right] =\mathbb{E%
}\left[ \left( \left\langle w\left( h\left( \mathbf{X}\right) \right)
,Z\right\rangle -Y\right) ^{2}\right] .
\end{equation*}%
Then Corollary \ref{Corollary uniform concentration} combined with Theorem %
\ref{Theorem Differentiation Seminorm bound} and (\ref{Derivativebound least
squares}) gives a high probability bound on 
\begin{equation*}
\sup_{h\in \mathcal{H}}\mathbb{E}\left[ f\left( h\left( \mathbf{X}^{\prime
}\right) \right) \right] -f\left( h\left( \mathbf{X}\right) \right) ,
\end{equation*}%
so as to justify the minimization of $f\left( h\left( \mathbf{X}\right)
\right) $ in $h$ if the Gaussian average $\mathbb{E}\left[ G\left( \mathcal{H%
}\left( \mathbf{X}\right) \right) \right] $ can be bounded.\bigskip

\section{Proof of Theorem \protect\ref{Theorem Main}}

We prove the theorem for $\mathcal{U}\subseteq \mathbb{R}$, the proof for $%
\mathcal{U}\subseteq \mathbb{R}^{d}$ being the same but with additional
notation. We take $f:\mathcal{U}^{n}\rightarrow \mathbb{R}$ as fixed for
this section and abbreviate $M=M_{Lip}\left( f\right) $ and $J=J_{Lip}\left(
f\right) $, when there is no ambiguity. We also use the following notation.
For any $i,j\in 
\mathbb{N}
$ we use $\left[ i,j\right] $ to denote the set of integers $\left[ i,j%
\right] =\left\{ i,...,j\right\} $ if $i\leq j$, or $\left[ i,j\right]
=\emptyset $ if $i>j$. Whenever two vectors in $\mathcal{X}^{n}$ or $%
\mathcal{U}^{n}$ are denoted $\mathbf{x}$ and $\mathbf{x}^{\prime }$, and $%
A\subseteq \left[ 1,n\right] $, then we use $\mathbf{x}^{A}$ to denote the
vector in $\mathcal{X}^{n}$ defined by%
\begin{equation}
x_{i}^{A}=\left\{ 
\begin{array}{ccc}
x_{i}^{\prime } & \text{if} & i\in A \\ 
x_{i} & \text{if} & i\notin A%
\end{array}%
\right. ,  \label{Define XhochA}
\end{equation}%
and we use $A^{c}$ to denote the complement of $A$ in $\left[ 1,n\right] $.
Also $\left\Vert .\right\Vert $ denotes the euclidean norm, either on $%
\mathbb{R}^{n}$ or $\mathbb{R}^{2n}$, depending on context, and $%
\left\langle .,.\right\rangle $ denotes the corresponding inner product.

We will use the following result about Gaussian processes, known as
Slepian's lemma (\cite{Boucheron13}, Theorem 13.3).

\begin{theorem}
\label{Slepian Lemma}Let $\Omega $ and $\Xi $ be mean zero Gaussian
processes indexed by a common finite set $\mathcal{H}$, such that%
\begin{equation*}
\mathbb{E}\left( \Omega _{h}-\Omega _{g}\right) ^{2}\leq \mathbb{E}\left(
\Xi _{h}-\Xi _{g}\right) ^{2}\text{ for all }h,g\in \mathcal{H}\text{.}
\end{equation*}%
Then%
\begin{equation*}
\mathbb{E}\sup_{h\in \mathcal{H}}\Omega _{h}\leq \mathbb{E}\sup_{h\in 
\mathcal{H}}\Xi _{h}.
\end{equation*}
\end{theorem}

The next lemma is the key to the way in which the interaction-seminorm $%
J=J_{Lip}\left( f\right) $ enters the proof.

\begin{lemma}
\label{Lemma JLipBound}For any $k\in \left[ 1,n\right] $ and $\mathbf{x},%
\mathbf{x}^{\prime }\in \mathcal{U}^{n}$ and $a,b\in \mathcal{U}$%
\begin{equation*}
D_{a,b}^{k}f\left( \mathbf{x}\right) -D_{a,b}^{k}f\left( \mathbf{x}^{\prime
}\right) \leq \frac{J}{n}\sum_{j:j\neq k}\left\vert x_{j}-x_{j}^{\prime
}\right\vert .
\end{equation*}
\end{lemma}

\begin{proof}
First assume $k=1$. Then%
\begin{eqnarray*}
D_{a,b}^{1}f\left( \mathbf{x}\right) -D_{a,b}^{1}f\left( \mathbf{x}^{\prime
}\right) &=&\sum_{j=2}^{n}D_{a,b}^{1}f\left( \mathbf{x}^{\left[ 1,j-1\right]
}\right) -D_{a,b}^{1}f\left( \mathbf{x}^{\left[ 1,j\right] }\right) \\
&=&\sum_{j=2}^{n}D_{a,b}^{1}D_{x_{j}x_{j}^{\prime }}^{j}f\left( \mathbf{x}^{%
\left[ 1,j\right] }\right) \leq \frac{J}{n}\sum_{j=2}^{n}\left\vert
x_{j}-x_{j}^{\prime }\right\vert .
\end{eqnarray*}%
If $k\neq 1$ let $f_{\pi }$ be the function $f_{\pi }\left( x\right)
=f\left( \pi x\right) $, where $\pi $ is the permutation exchanging the
first and the $k$-th argument, observe that $J_{Lip}\left( f_{\pi }\right)
=J_{Lip}\left( f\right) =J$, and apply the above to $f_{\pi }$. \bigskip
\bigskip
\end{proof}

For $k\in \left\{ 1,...,n\right\} $ define a function $F_{k}:\mathcal{U}%
^{2n}\rightarrow \mathbb{R}$ by%
\begin{equation*}
F_{k}\left( \mathbf{x,x}^{\prime }\right) =\frac{1}{2^{k}}\sum_{A\subseteq %
\left[ 1,k-1\right] }\left( D_{x_{k},x_{k}^{\prime }}^{k}f\left( \mathbf{x}%
^{A}\right) +D_{x_{k},x_{k}^{\prime }}^{k}f\left( \mathbf{x}^{A^{c}}\right)
\right) .
\end{equation*}%
$F_{k}\left( \mathbf{x},\mathbf{x}^{\prime }\right) $ changes sign if we
exchange $x_{k}$ and $x_{k}^{\prime }$, but if $i<k$, then $i\in \left[ 1,k-1%
\right] $, so the exchange of $x_{i}$ and $x_{i}^{\prime }$ exchanges just
terms in the above sum (see (\ref{Elementary fact}) in Appendix \ref%
{Appendix Proof of Lemma}) and therefore leaves $F_{k}\left( \mathbf{x},%
\mathbf{x}^{\prime }\right) $ invariant. This is the reason why we use the
somewhat complicated representation of $f\left( \mathbf{x}\right) -f\left( 
\mathbf{x}^{\prime }\right) $, as given by the next lemma.\bigskip

\begin{lemma}
\label{Lemma F_K Representation}For $\mathbf{x},\mathbf{x}^{\prime }\in 
\mathcal{U}^{n}$ we have%
\begin{equation*}
f\left( \mathbf{x}\right) -f\left( \mathbf{x}^{\prime }\right)
=\sum_{k=1}^{n}F_{k}\left( \mathbf{x},\mathbf{x}^{\prime }\right) .
\end{equation*}
\end{lemma}

The proof is given in Appendix \ref{Appendix Proof of Lemma}.\bigskip

For $\left( \mathbf{x,x}^{\prime }\right) \in \mathcal{U}^{2n}$ and $k\in %
\left[ 1,n\right] $ we define a vector $v^{k}\left( \mathbf{x,x}^{\prime
}\right) \in \mathbb{R}^{2n}$ by%
\begin{equation*}
v_{i}^{k}\left( \mathbf{x,x}^{\prime }\right) =\left\{ 
\begin{array}{lcl}
2Mx_{k} & \text{if} & i=k \\ 
Jn^{-1/2}x_{i} & \text{if} & i\neq k,i\leq n \\ 
2Mx_{k}^{\prime } & \text{if} & i=n+k \\ 
Jn^{-1/2}x_{i-n}^{\prime } & \text{if} & i\neq n+k,i>n%
\end{array}%
\right. .
\end{equation*}%
\bigskip

\begin{lemma}
\label{Lemma F_K Difference}For $\left( \mathbf{x,x}^{\prime }\right)
,\left( \mathbf{y,y}^{\prime }\right) \in \mathcal{U}^{2n}$ and $k\in \left[
1,n\right] $ we have%
\begin{equation*}
F_{k}\left( \mathbf{x},\mathbf{x}^{\prime }\right) -F_{k}\left( \mathbf{y},%
\mathbf{y}^{\prime }\right) \leq \sqrt{\pi /2}~\mathbb{E}\left\vert
\left\langle \mathbf{\gamma },v^{k}\left( \mathbf{x,x}^{\prime }\right)
-v^{k}\left( \mathbf{y,y}^{\prime }\right) \right\rangle \right\vert ,
\end{equation*}%
where $\mathbf{\gamma }=\left( \gamma _{1},...,\gamma _{n},\gamma
_{1}^{\prime },...,\gamma _{n}^{\prime }\right) $ is a vector of $2n$
independent standard normal variables.
\end{lemma}

\begin{proof}
Using the definition of $M=M_{Lip}\left( f\right) $ and Lemma \ref{Lemma
JLipBound} we have for any $A\subseteq \left\{ 1,...,n\right\} $ 
\begin{align*}
D_{x_{k},x_{k}^{\prime }}^{k}f\left( \mathbf{x}^{A}\right)
-D_{y_{k},y_{k}^{\prime }}^{k}f\left( \mathbf{y}^{A}\right) &
=D_{x_{k},y_{k}}^{k}f\left( \mathbf{x}^{A}\right) +D_{y_{k}^{\prime
},x_{k}^{\prime }}^{k}f\left( \mathbf{x}^{A}\right) +D_{y_{k},y_{k}^{\prime
}}^{k}\left( f\left( \mathbf{x}^{A}\right) -f\left( \mathbf{y}^{A}\right)
\right)  \\
& \leq M\left( \left\vert x_{k}-y_{k}\right\vert +\left\vert x_{k}^{\prime
}-y_{k}^{\prime }\right\vert \right) +\frac{J}{n}\sum_{i:i\neq k}\left\vert
x_{i}^{A}-y_{i}^{A}\right\vert 
\end{align*}%
Define vectors $u$, $w\in \mathbb{R}^{2n}$ by $u_{i}=\left\vert
v_{i}^{k}\left( \mathbf{x,x}^{\prime }\right) -v_{i}^{k}\left( \mathbf{y,y}%
^{\prime }\right) \right\vert $ and $w_{i}=1/2$ if $i=k$ or $i=n+k$ and $%
w_{i}=1/\left( 2\sqrt{n}\right) $ otherwise. Then $\left\Vert w\right\Vert
\leq 1$ and%
\begin{align*}
& F_{k}\left( \mathbf{x},\mathbf{x}^{\prime }\right) -F_{k}\left( \mathbf{y},%
\mathbf{y}^{\prime }\right)  \\
& =\frac{1}{2^{k}}\sum_{A\subseteq \left[ 1,k-1\right] }\left(
D_{x_{k},x_{k}^{\prime }}^{k}f\left( \mathbf{x}^{A}\right)
-D_{y_{k},y_{k}^{\prime }}^{k}f\left( \mathbf{y}^{A}\right)
+D_{x_{k},x_{k}^{\prime }}^{k}f\left( \mathbf{x}^{A^{c}}\right)
-D_{y_{k},y_{k}^{\prime }}^{k}f\left( \mathbf{y}^{A^{c}}\right) \right)  \\
& \leq \frac{1}{2^{k}}\sum_{A\subseteq \left[ 1,k-1\right] }\left( 2M\left(
\left\vert x_{k}-y_{k}\right\vert +\left\vert x_{k}^{\prime }-y_{k}^{\prime
}\right\vert \right) +\frac{J}{n}\sum_{i:i\neq k}\left\vert
x_{i}^{A}-y_{i}^{A}\right\vert +\frac{J}{n}\sum_{i:i\neq k}\left\vert
x_{i}^{A^{c}}-y_{i}^{A^{c}}\right\vert \right)  \\
& =M\left( \left\vert x_{k}-y_{k}\right\vert +\left\vert x_{k}^{\prime
}-y_{k}^{\prime }\right\vert \right) +\frac{J}{2n}\sum_{i\neq k}\left(
\left\vert x_{i}-y_{i}\right\vert +\left\vert x_{i}^{\prime }-y_{i}^{\prime
}\right\vert \right) =\left\langle w,u\right\rangle  \\
& \leq \left\Vert u\right\Vert =\left\Vert v^{k}\left( \mathbf{x,x}^{\prime
}\right) -v^{k}\left( \mathbf{y,y}^{\prime }\right) \right\Vert =\sqrt{\pi /2%
}~\mathbb{E}\left\vert \left\langle \mathbf{\gamma },v^{k}\left( \mathbf{x,x}%
^{\prime }\right) -v^{k}\left( \mathbf{y,y}^{\prime }\right) \right\rangle
\right\vert ,
\end{align*}%
where we used Cauchy-Schwarz and a standard formula following from rotation
invariance of the isotropic normal distribution. \bigskip 
\end{proof}

\begin{proof}
(Proof of Theorem \ref{Theorem Main}) With $\mathbf{X}^{\prime }$
identically distributed to $\mathbf{X}$ we have 
\begin{equation*}
\mathbb{E}\sup_{h}\mathbb{E}_{\mathbf{X}}\left[ f\left( h\left( \mathbf{X}%
\right) \right) \right] -f\left( h\left( \mathbf{X}^{\prime }\right) \right)
\leq \mathbb{E}\sup_{h}f\left( h\left( \mathbf{X}\right) \right) -f\left(
h\left( \mathbf{X}^{\prime }\right) \right) ,
\end{equation*}%
so it suffices to bound the right hand side above. We first prove that%
\begin{equation}
\mathbb{E}\sup_{h}f\left( h\left( \mathbf{X}\right) \right) -f\left( h\left( 
\mathbf{X}^{\prime }\right) \right) \leq \sqrt{\pi /2}~\mathbb{E}_{\mathbf{XX%
}^{\prime }}\mathbb{E}_{\mathbf{\gamma }}\sup_{h}\sum_{k=1}^{n}\left\langle 
\mathbf{\gamma }_{k},v^{k}\left( h\left( \mathbf{X}\right) ,h\left( \mathbf{X%
}^{\prime }\right) \right) \right\rangle ,  \label{Claim}
\end{equation}%
where the $\mathbf{\gamma }_{k}$ are independent copies of the vector $%
\mathbf{\gamma }$ in Lemma \ref{Lemma F_K Difference}. Then we use Slepian's
inequality to bound the right hand side above.

To prove (\ref{Claim}) we show by induction on $m\in \left\{ 0,...,n\right\} 
$ that%
\begin{multline*}
\mathbb{E}\sup_{h}f\left( h\left( \mathbf{X}\right) \right) -f\left( h\left( 
\mathbf{X}^{\prime }\right) \right)  \\
\leq \mathbb{E}_{\mathbf{XX}^{\prime }}\mathbb{E}_{\mathbf{\gamma }}\left[
\sup_{h}\sqrt{\pi /2}\sum_{k=1}^{m}\left\langle \mathbf{\gamma }%
_{k},v^{k}\left( h\left( \mathbf{X}\right) ,h\left( \mathbf{X}^{\prime
}\right) \right) \right\rangle +\sum_{k=m+1}^{n}F_{k}\left( h\left( \mathbf{X%
}\right) ,h\left( \mathbf{X}^{\prime }\right) \right) \right] .
\end{multline*}%
For $m=n$ this is (\ref{Claim}), and for $m=0$ it is just Lemma \ref{Lemma
F_K Representation}. Suppose it holds for $m-1$, with some $m\leq n$, and
define for each $h\in \mathcal{H}$ a real valued random variable $R_{h}$ by 
\begin{equation*}
R_{h}=\sqrt{\pi /2}\sum_{k=1}^{m-1}\left\langle \mathbf{\gamma }%
_{k},v^{k}\left( h\left( \mathbf{X}\right) ,h\left( \mathbf{X}^{\prime
}\right) \right) \right\rangle +\sum_{k=m+1}^{n}F_{k}\left( h\left( \mathbf{X%
}\right) ,h\left( \mathbf{X}^{\prime }\right) \right) .
\end{equation*}%
The expectation $\mathbb{E=E}_{\mathbf{XX}^{\prime }}\mathbb{E}_{\mathbf{%
\gamma }}\left[ .\right] $ is invariant under the simultaneous exchange of $%
X_{m}$ and $X_{m}^{\prime }$ and, for all $k<m$, of $\gamma _{km}$ and $%
\gamma _{km}^{\prime }$, which leaves $R_{h}$ invariant but changes the sign
of $F_{m}$. Using this fact and the induction assumption%
\begin{align*}
& \mathbb{E}\sup_{h}f\left( h\left( \mathbf{X}\right) \right) -f\left(
h\left( \mathbf{X}^{\prime }\right) \right)  \\
& \leq \mathbb{E}\sup_{h}F_{m}\left( h\left( \mathbf{X}\right) ,h\left( 
\mathbf{X}^{\prime }\right) \right) +R_{h} \\
& =\frac{1}{2}\mathbb{E}\sup_{h,g}F_{m}\left( h\left( \mathbf{X}\right)
,h\left( \mathbf{X}^{\prime }\right) \right) -F_{m}\left( g\left( \mathbf{X}%
\right) ,g\left( \mathbf{X}^{\prime }\right) \right) +R_{h}+R_{g}
\end{align*}%
Using Lemma \ref{Lemma F_K Difference}, with $\left( \mathbf{x,x}^{\prime
}\right) $ replaced by $\left( h\left( \mathbf{X}\right) ,h\left( \mathbf{X}%
^{\prime }\right) \right) $ and $\left( \mathbf{y,y}^{\prime }\right) $
replaced by $\left( g\left( \mathbf{X}\right) ,g\left( \mathbf{X}^{\prime
}\right) \right) $, we get%
\begin{align*}
& \mathbb{E}\sup_{h}f\left( h\left( \mathbf{X}\right) \right) -f\left(
h\left( \mathbf{X}^{\prime }\right) \right)  \\
& \leq \frac{1}{2}\mathbb{E}\sup_{h,g}\sqrt{\pi /2}~\mathbb{E}_{\gamma
_{m}}\left\vert \left\langle \mathbf{\gamma }_{m},v^{m}\left( h\left( 
\mathbf{X}\right) ,h\left( \mathbf{X}^{\prime }\right) \right) -v^{m}\left(
g\left( \mathbf{X}\right) ,g\left( \mathbf{X}^{\prime }\right) \right)
\right\rangle \right\vert +R_{h}+R_{g} \\
& \leq \frac{1}{2}\mathbb{E}\sup_{h,g}\sqrt{\pi /2}\left\vert \left\langle 
\mathbf{\gamma }_{m},v^{m}\left( h\left( \mathbf{X}\right) ,h\left( \mathbf{X%
}^{\prime }\right) \right) \right\rangle -\left\langle \mathbf{\gamma }%
_{m},v^{m}\left( g\left( \mathbf{X}\right) ,g\left( \mathbf{X}^{\prime
}\right) \right) \right\rangle \right\vert +R_{h}+R_{g} \\
& =\frac{1}{2}\mathbb{E}\sup_{h,g}\sqrt{\pi /2}\left\langle \mathbf{\gamma }%
_{m},v^{m}\left( h\left( \mathbf{X}\right) ,h\left( \mathbf{X}^{\prime
}\right) \right) \right\rangle -\sqrt{\pi /2}\left\langle \mathbf{\gamma }%
_{m},v^{m}\left( g\left( \mathbf{X}\right) ,g\left( \mathbf{X}^{\prime
}\right) \right) \right\rangle +R_{h}+R_{g}.
\end{align*}%
Here we could drop the absolute value because the supremum is in both $h$
and $g$, and the remaining sum is invariant under the exchange of $h$ and $g$%
. The symmetry of the standard normal distribution then gives%
\begin{eqnarray*}
\mathbb{E}\sup_{h}f\left( h\left( \mathbf{X}\right) \right) -f\left( h\left( 
\mathbf{X}^{\prime }\right) \right)  &\leq &\frac{1}{2}\mathbb{E}\sup_{h}%
\sqrt{\pi /2}\left\langle \mathbf{\gamma }_{m},v^{m}\left( h\left( \mathbf{X}%
\right) ,h\left( \mathbf{X}^{\prime }\right) \right) \right\rangle +R_{h} \\
&&+\frac{1}{2}\mathbb{E}\sup_{g}\sqrt{\pi /2}\left\langle -\mathbf{\gamma }%
_{m},v^{m}\left( g\left( \mathbf{X}\right) ,g\left( \mathbf{X}^{\prime
}\right) \right) \right\rangle +R_{g} \\
&=&\mathbb{E}\sup_{h}\sqrt{\pi /2}\left\langle \mathbf{\gamma }%
_{m},v^{m}\left( h\left( \mathbf{X}\right) ,h\left( \mathbf{X}^{\prime
}\right) \right) \right\rangle +R_{h}.
\end{eqnarray*}%
By definition of $R_{h}$ this completes the induction and proves the claim (%
\ref{Claim}).

We now condition on $\mathbf{X}$ and $\mathbf{X}^{\prime }$ and seek to
bound 
\begin{equation*}
\mathbb{E}_{\gamma }\sup_{h}\sum_{k=1}^{n}\left\langle \mathbf{\gamma }%
_{k},v^{k}\left( h\left( \mathbf{x}\right) ,h\left( \mathbf{x}^{\prime
}\right) \right) \right\rangle =\mathbb{E}_{\gamma }\sup_{h}\Omega _{h},
\end{equation*}%
where $\Omega $ is the Gaussian process indexed by $\mathcal{H}$ 
\begin{equation*}
\Omega _{h}=\sum_{k=1}^{n}\left\langle \mathbf{\gamma }_{k},v^{k}\left(
h\left( \mathbf{x}\right) ,h\left( \mathbf{x}^{\prime }\right) \right)
\right\rangle 
\end{equation*}%
Now we have%
\begin{eqnarray*}
\mathbb{E}\left[ \left( \Omega _{h}-\Omega _{g}\right) ^{2}\right] 
&=&\sum_{k=1}^{n}\left\Vert v^{k}\left( h\left( \mathbf{x}\right) ,h\left( 
\mathbf{x}^{\prime }\right) \right) -v^{k}\left( g\left( \mathbf{x}\right)
,g\left( \mathbf{x}^{\prime }\right) \right) \right\Vert ^{2} \\
&=&\sum_{k=1}^{n}\left( 4M^{2}\left( \left\vert h\left( x_{k}\right)
-g\left( x_{k}\right) \right\vert ^{2}+\left\vert h\left( x_{k}^{\prime
}\right) -g\left( x_{k}^{\prime }\right) \right\vert ^{2}\right) \right. + \\
&&\text{ \ \ \ \ \ \ \ \ \ }\left. +\frac{J^{2}}{n}\sum_{i:i\neq k}\left(
\left\vert h\left( x_{i}\right) -g\left( x_{i}\right) \right\vert
^{2}+\left\vert h\left( x_{i}^{\prime }\right) -g\left( x_{i}^{\prime
}\right) \right\vert ^{2}\right) \right)  \\
&\leq &\left( 4M^{2}+J^{2}\right) \sum_{k=1}^{n}\left( \left\vert h\left(
x_{k}\right) -g\left( x_{k}\right) \right\vert ^{2}+\left\vert h\left(
x_{k}^{\prime }\right) -g\left( x_{k}^{\prime }\right) \right\vert
^{2}\right)  \\
&=&\mathbb{E}\left[ \left( \Xi _{h}-\Xi _{g}\right) ^{2}\right] 
\end{eqnarray*}%
where $\Xi $ is the Gaussian process 
\begin{equation*}
\Xi _{h}=\sqrt{4M^{2}+J^{2}}\sum_{k=1}^{n}\left( \gamma _{k}h\left(
x_{k}\right) +\gamma _{k}^{\prime }h\left( x_{k}^{\prime }\right) \right) 
\end{equation*}%
It follows from Slepian's inequality (Theorem \ref{Slepian Lemma}) that $%
\mathbb{E}\sup_{h}\Omega _{h}\leq \mathbb{E}\sup_{h}\Xi _{h}.$ Combined with
(\ref{Claim}) this gives%
\begin{eqnarray*}
\mathbb{E}\sup_{h}f\left( h\left( \mathbf{X}\right) \right) -f\left( h\left( 
\mathbf{X}^{\prime }\right) \right)  &\leq &\sqrt{\pi /2}\mathbb{E}_{\mathbf{%
XX}^{\prime }}\mathbb{E}_{\mathbf{\gamma }}\sup_{h}\Omega _{h}\leq \sqrt{\pi
/2}\mathbb{E}_{\mathbf{XX}^{\prime }}\mathbb{E}_{\mathbf{\gamma }%
}\sup_{h}\Xi _{h} \\
&=&\sqrt{\pi /2}\sqrt{4M^{2}+J^{2}}\mathbb{E}_{\mathbf{XX}^{\prime }}\mathbb{%
E}_{\mathbf{\gamma }}\sup_{h}\sum_{k=1}^{n}\gamma _{k}h\left( X_{k}\right)
+\gamma _{k}^{\prime }h\left( X_{k}^{\prime }\right)  \\
&\leq &\sqrt{2\pi }\left( 2M+J\right) ~\mathbb{E}G\left( \mathcal{H}\left( 
\mathbf{X}\right) \right) .
\end{eqnarray*}
\end{proof}

\appendix

\section{The bounded difference inequality}

\begin{theorem}
\label{Theorem bounded difference}(\cite{McDiarmid 1998} or \cite%
{Boucheron13}) Suppose $f:\mathcal{X}^{n}\rightarrow \mathbb{R}$ and $%
\mathbf{X}=\left( X_{1},...,X_{n}\right) $ is a vector of independent random
variables with values in $\mathcal{X}$, $\mathbf{X}^{\prime }$ is iid to $%
\mathbf{X}$. Then 
\begin{equation*}
\Pr \left\{ f\left( \mathbf{X}\right) -\mathbb{E}f\left( \mathbf{X}^{\prime
}\right) >t\right\} \leq \exp \left( \frac{-2t^{2}}{\sup_{\mathbf{x}\in 
\mathcal{X}^{n}}\sum_{k}\sup_{y,y^{\prime }\in \mathcal{X}}\left(
D_{y,y^{\prime }}^{k}f\left( \mathbf{x}\right) \right) ^{2}}\right) .
\end{equation*}%
\bigskip 
\end{theorem}

\section{U- and V-statistics\label{Appendix Proof Ustat}}

We prove Theorem \ref{Theorem U-statistic}.

\begin{proof}
\begin{eqnarray*}
D_{y,y^{\prime }}^{k}V\left( \mathbf{x}\right)  &\leq &n^{-m}\sum_{\mathbf{j}%
\in \left( 
\mathbb{N}
_{n}\right) ^{m},\exists i\in 
\mathbb{N}
_{m}\text{, }j_{i}=k}D_{y,y^{\prime }}^{i}\kappa _{\mathbf{j}}\left(
x_{j_{1}},...,x_{j_{m}}\right)  \\
&\leq &n^{-m}\sum_{\mathbf{j}\in \left( 
\mathbb{N}
_{n}\right) ^{m},\exists i\in 
\mathbb{N}
_{m}\text{, }j_{i}=k}M\left( \kappa _{\mathbf{j}}\right) \text{.}
\end{eqnarray*}%
But $n^{-m}\left\vert \left\{ \mathbf{j}\in \left( 
\mathbb{N}
_{n}\right) ^{m},\exists i\in 
\mathbb{N}
_{m}\text{, }j_{i}=k\right\} \right\vert =m/n$. So $M\left( f\right) \leq
m\max_{_{\mathbf{j}}}M\left( \kappa _{\mathbf{j}}\right) /n$, with exactly
the same argument for $M_{Lip}\left( f\right) $. Also%
\begin{eqnarray*}
&&D_{z,z^{\prime }}^{l}D_{y,y^{\prime }}^{k}V\left( \mathbf{x}\right)  \\
&\leq &n^{-m}\sum_{\substack{ \mathbf{j}\in \left( 
\mathbb{N}
_{n}\right) ^{m},\exists i,i^{\prime }\in 
\mathbb{N}
_{m}\text{,} \\ \text{ }j_{i}=k,j_{i^{\prime }}=l}}D_{z,z^{\prime
}}^{i^{\prime }}D_{y,y^{\prime }}^{i}\kappa _{\mathbf{j}}\left(
x_{j_{1}},...,x_{j_{m}}\right)  \\
&\leq &n^{-m}\sum_{\substack{ \mathbf{j}\in \left( 
\mathbb{N}
_{n}\right) ^{m},\exists i,i^{\prime }\in 
\mathbb{N}
_{m}\text{,} \\ \text{ }j_{i}=k,j_{i^{\prime }}=l}}J\left( \kappa _{\mathbf{j%
}}\right) \text{.}
\end{eqnarray*}%
But 
\begin{equation*}
n^{-m}\left\vert \left\{ \mathbf{j}\in \left( 
\mathbb{N}
_{n}\right) ^{m}:\exists i,i^{\prime }\in 
\mathbb{N}
_{m}\text{, }j_{i}=k,j_{i^{\prime }}=l\right\} \right\vert \leq m^{2}/n^{2}%
\text{.}
\end{equation*}%
So $J\left( f\right) \leq m^{2}\max_{_{\mathbf{j}}}J\left( \kappa _{\mathbf{j%
}}\right) /n$, with exactly the same argument for $J_{Lip}\left( f\right) $.
This completes proof for V-statistics. For the case of U-statistics we have
to count the number of subsets $S\subseteq 
\mathbb{N}
_{n}$ of cardinality $m$ containing a fixed $k\in 
\mathbb{N}
_{n}$ or two distict $k$,$l\in 
\mathbb{N}
_{n}$ respectively. This is $\binom{n-1}{m-1}$ or $\binom{n-2}{m-2}$
respectively and 
\begin{equation*}
\frac{\binom{n-1}{m-1}}{\binom{n}{m}}=\frac{m!\left( n-1\right) !}{n!\left(
m-1\right) !}=\frac{m}{n}\text{ or }\frac{\binom{n-2}{m-2}}{\binom{n}{m}}=%
\frac{m!\left( n-2\right) !}{n!\left( m-2\right) !}=\frac{m\left( m-1\right) 
}{n\left( n-1\right) }\leq \frac{m^{2}}{n^{2}}.
\end{equation*}%
\bigskip 
\end{proof}

\section{Differentiation\label{Appendix proof differentiation}}

We prove Theorem \ref{Theorem Differentiation Seminorm bound}\bigskip

\begin{proof}
Fix $\mathbf{x}\in \mathcal{U}^{n}$, $y,y^{\prime },z,z^{\prime }\in 
\mathcal{X}$ and $k\neq l\in 
\mathbb{N}
_{n}$. For $0\leq s,t\leq 1$ define $\mathbf{x}\left( t\right) =S_{y^{\prime
}+t\left( y-y^{\prime }\right) }^{k}\mathbf{x}$ and $\mathbf{x}\left(
s,t\right) =S_{z^{\prime }+t\left( z-z^{\prime }\right) }^{l}S_{y^{\prime
}+t\left( y-y^{\prime }\right) }^{k}\mathbf{x}$. Convexity insures that $%
f\left( \mathbf{x}\left( t\right) \right) $ and $f\left( \mathbf{x}\left(
s,t\right) \right) $ is defined for all values of $s$ and $t$. Then%
\begin{eqnarray*}
D_{yy^{\prime }}^{k}f\left( \mathbf{x}\right) &=&f\left( \mathbf{x}\left(
1\right) \right) -f\left( \mathbf{x}\left( 0\right) \right)
=\int_{0}^{1}\left\langle \partial _{k}f\left( \mathbf{x}\left( t\right)
\right) ,y-y^{\prime }\right\rangle dt \\
&\leq &\left\Vert \partial _{k}f\right\Vert \left\Vert y-y^{\prime
}\right\Vert .
\end{eqnarray*}%
Similarly%
\begin{eqnarray*}
D_{zz^{\prime }}^{l}D_{yy^{\prime }}^{k}f\left( \mathbf{x}\right) &=&\left(
f\left( \mathbf{x}\left( 1,1\right) \right) -f\left( \mathbf{x}\left(
1,0\right) \right) \right) -\left( f\left( \mathbf{x}\left( 0,1\right)
\right) -f\left( \mathbf{x}\left( 0,0\right) \right) \right) \\
&=&\int_{0}^{1}\int_{0}^{1}\left\langle \partial _{lk}f\left( \mathbf{x}%
\left( s,t\right) \right) \left( z-z^{\prime }\right) ,\left( y-y^{\prime
}\right) \right\rangle dsdt\leq \left\Vert \partial _{lk}f\right\Vert
\left\Vert y-y^{\prime }\right\Vert \left\Vert z-z^{\prime }\right\Vert \\
&\leq &\Delta \left\Vert \partial _{lk}f\right\Vert \left\Vert y-y^{\prime
}\right\Vert \text{.}
\end{eqnarray*}%
\bigskip
\end{proof}

\section{Proof of Lemma \protect\ref{Lemma F_K Representation}\label%
{Appendix Proof of Lemma}}

We will use the following elementary fact: if $\Phi $ is a function defined
on subsets of $\left[ 1,k-1\right] $, then for every $i\in \left[ 1,k-1%
\right] $%
\begin{equation}
\sum_{A\subseteq \left[ 1,k-1\right] }\Phi \left( A\right) =\sum_{A\subseteq %
\left[ 1,k-1\right] \backslash \left[ i\right] }\left( \Phi \left( A\right)
+\Phi \left( A\cup \left[ i\right] \right) \right) .  \label{Elementary fact}
\end{equation}%
Also note that%
\begin{equation*}
\sum_{A\subseteq \left[ 1,k-1\right] }D_{x_{k},x_{k}^{\prime }}^{k}f\left( 
\mathbf{x}^{A^{c}}\right) =\sum_{A\subseteq \left[ 1,k-1\right]
}D_{x_{k},x_{k}^{\prime }}^{k}f\left( \mathbf{x}^{A\cup \left[ k,n\right]
}\right) ,
\end{equation*}%
so it suffices to prove the following:

\textbf{Claim: }For every set $\mathcal{X}$, all $n\in 
\mathbb{N}
$, all functions $f:\mathcal{X}^{n}\rightarrow \mathbb{R}$ and all vectors $%
\mathbf{x}$ and $\mathbf{x}^{\prime }\in \mathcal{X}^{n}$ we have 
\begin{equation*}
f\left( \mathbf{x}\right) -f\left( \mathbf{x}^{\prime }\right)
=\sum_{k=1}^{n}\frac{1}{2^{k}}\sum_{A\subseteq \left[ 1,k-1\right] }\left(
D_{x_{k},x_{k}^{\prime }}^{k}f\left( \mathbf{x}^{A}\right)
+D_{x_{k},x_{k}^{\prime }}^{k}f\left( \mathbf{x}^{A\cup \left[ k,n\right]
}\right) \right) \text{.}
\end{equation*}

\begin{proof}
By induction on $n$. Since the empty set is the only subset of $\left[ 1,0%
\right] $, the case $n=1$ reduces to the identity%
\begin{equation*}
f\left( \mathbf{x}\right) -f\left( \mathbf{x}^{\prime }\right) =\frac{1}{2}%
\left( f\left( \mathbf{x}\right) -f\left( \mathbf{x}^{\prime }\right)
+f\left( \mathbf{x}\right) -f\left( \mathbf{x}^{\prime }\right) \right) .
\end{equation*}%
Assume the claim to be true for $n-1$ and let $f$, $\mathbf{x}$ and $\mathbf{%
x}^{\prime }$ be as in the statement of the claim. Let $\mathbf{z}$ and $%
\mathbf{z}^{\prime }$ be the $\left( n-1\right) $-dimensional vectors $%
\left( x_{2},...,x_{n}\right) $ and $\left( x_{2}^{\prime
},...,x_{n}^{\prime }\right) $ respectively and define $g:\mathcal{X}%
^{n-1}\rightarrow \mathbb{R}$ by $g\left( \mathbf{z}\right) =f\left( x_{1},%
\mathbf{z}\right) $. By the induction assumption applied to $g$ and a change
of variables%
\begin{eqnarray*}
f\left( \mathbf{x}\right) -f\left( \mathbf{x}^{\left[ 2,n\right] }\right) 
&=&g\left( \mathbf{z}\right) -g\left( \mathbf{z}^{\prime }\right)  \\
&=&\sum_{k=1}^{n-1}\frac{1}{2^{k}}\sum_{A\subseteq \left[ 1,k-1\right]
}D_{z_{k},z_{k}^{\prime }}^{k}g\left( \mathbf{z}^{A}\right)
+D_{z_{k},z_{k}^{\prime }}^{k}g\left( \mathbf{z}^{A\cup \left[ k,n-1\right]
}\right)  \\
&=&\sum_{k=2}^{n}\frac{1}{2^{k-1}}\sum_{A\subseteq \left[ 2,k-1\right]
}D_{x_{k},x_{k}^{\prime }}^{k}f\left( \mathbf{x}^{A}\right)
+D_{x_{k},x_{k}^{\prime }}^{k}f\left( \mathbf{x}^{A\cup \left[ k,n\right]
}\right) 
\end{eqnarray*}%
In the same way, replacing $x_{1}$ by $x_{1}^{\prime }$ in the definition of 
$g$,%
\begin{equation*}
f\left( \mathbf{x}^{\left[ 1\right] }\right) -f\left( \mathbf{x}^{\left[ 1,n%
\right] }\right) =\sum_{k=2}^{n}\frac{1}{2^{k-1}}\sum_{A\subseteq \left[
2,k-1\right] }D_{x_{k},x_{k}^{\prime }}^{k}f\left( \mathbf{x}^{\left[ 1%
\right] \cup A}\right) +D_{x_{k},x_{k}^{\prime }}^{k}f\left( \mathbf{x}^{%
\left[ 1\right] \cup A\cup \left[ k,n\right] }\right) .
\end{equation*}%
Thus, adding and subtracting $f\left( \mathbf{x}^{\left[ 1\right] }\right) /2
$ and $f\left( \mathbf{x}^{\left[ 2,n\right] }\right) /2$ from $f\left( 
\mathbf{x}\right) -f\left( \mathbf{x}^{\prime }\right) $, we obtain%
\begin{align*}
& f\left( \mathbf{x}\right) -f\left( \mathbf{x}^{\prime }\right)  \\
& =\frac{1}{2}\left( f\left( \mathbf{x}\right) -f\left( \mathbf{x}^{\left[ 1%
\right] }\right) +f\left( \mathbf{x}^{\left[ 2,n\right] }\right) -f\left( 
\mathbf{x}^{\left[ 1,n\right] }\right) \right) + \\
& \text{ \ \ \ \ \ \ }+\frac{1}{2}\left( f\left( \mathbf{x}\right) -f\left( 
\mathbf{x}^{\left[ 2,n\right] }\right) +f\left( \mathbf{x}^{\left[ 1\right]
}\right) -f\left( \mathbf{x}^{\left[ 1,n\right] }\right) \right)  \\
& =\frac{1}{2}\left( D_{x_{1}x_{1}^{\prime }}^{1}f\left( \mathbf{x}\right)
+D_{x_{1}x_{1}^{\prime }}^{1}f\left( \mathbf{x}^{\left[ 1,n\right] }\right)
\right) + \\
& \text{ \ \ \ \ \ \ }+\frac{1}{2}\sum_{k=2}^{n}\frac{1}{2^{k-1}}%
\sum_{A\subseteq \left[ 2,k-1\right] }\left( D_{x_{k},x_{k}^{\prime
}}^{k}f\left( \mathbf{x}^{A}\right) +D_{x_{k},x_{k}^{\prime }}^{k}f\left( 
\mathbf{x}^{A\cup \left[ k,n\right] }\right) \right. + \\
& \text{ \ \ \ \ \ \ \ \ \ \ \ \ \ \ \ \ \ \ \ \ \ \ \ \ \ \ \ \ \ \ \ \ \ \
\ }+\left. D_{x_{k},x_{k}^{\prime }}^{k}f\left( \mathbf{x}^{\left[ 1\right]
\cup A}\right) +D_{x_{k},x_{k}^{\prime }}^{k}f\left( \mathbf{x}^{\left[ 1%
\right] \cup A\cup \left[ k,n\right] }\right) \right)  \\
& =\frac{1}{2}\left( D_{x_{1}x_{1}^{\prime }}^{1}f\left( \mathbf{x}\right)
+D_{x_{1}x_{1}^{\prime }}^{1}f\left( \mathbf{x}^{\left[ 1,n\right] }\right)
\right) + \\
& \text{ \ \ \ \ \ \ }+\sum_{k=2}^{n}\frac{1}{2^{k}}\sum_{A\subseteq \left[
1,k-1\right] }\left( D_{x_{k},x_{k}^{\prime }}^{k}f\left( \mathbf{x}%
^{A}\right) +D_{x_{k},x_{k}^{\prime }}^{k}f\left( \mathbf{x}^{A\cup \left[
k,n\right] }\right) \right) ,
\end{align*}%
The last identity used (\ref{Elementary fact}) with $i=1$. This completes
the induction.
\end{proof}


\begin{thebibliography}{Bartlett and Mendelson(2002)}
\bibitem[Agarwal et al.(2005)]{Agarwal 2005} S. Agarwal, T. Graepel, R.
Herbrich, S. Har-Peled, S., and D. Roth, D. Generalization bounds for the
area under the ROC curve. \emph{Journal of Machine Learning Research},
6:393-425, 2005.

\bibitem[Ambroladze et al.(2007)]{Ambroladze 05} A. Ambroladze, E.
Parrado-Hern\'{a}ndez, and J. Shawe-Taylor. Complexity of pattern classes
and the Lipschitz property. \emph{Theoretical Computer Science},
382(3):232--246, 2007.

\bibitem[Bartlett and Mendelson(2002)]{Bartlett 2002} P. L. Bartlett and S.
Mendelson. Rademacher and Gaussian complexities: risk bounds and structural
results. \textit{Journal of Machine Learning Research}, 3: 463--482, 2002.

\bibitem[Biau et al.(2008)]{Biau Lugosi} G. Biau, L. Devroye, and G. Lugosi.
On the performance of clustering in Hilbert spaces. \emph{IEEE Transactions
on Information Theory}, 54(2):781--790, 2008.

\bibitem[Boucheron et al.(2013)]{Boucheron13} S. Boucheron, G. Lugosi, P.
Massart. \emph{Concentration Inequalities}, Oxford University Press, 2013.

\bibitem[Cao et al.(2016)]{Cao 2016} Q. Cao, Z. C. Guo, and Y. Ying.
Generalization bounds for metric and similarity learning. \emph{Machine
Learning}, 102(1):115--132, 2016.

\bibitem[Clemencon et al.(2008)]{Clemencon 2008} S. Cl\'{e}men\c{c}on, G.
Lugosi, and N. Vayatis. Ranking and empirical minimization of U-statistics. 
\emph{The Annals of Statistics}, 36(2):844--874, 2008.

\bibitem[Cuesta-Albertos et al.(1997)]{Cuesta 1997} J. A. Cuesta-Albertos,
A. Gordaliza, and C. Matr\'{a}n, C. Trimmed $k$-means: An attempt to
robustify quantizers. \emph{The Annals of Statistics}, 25(2):553--576, 1997.

\bibitem[Garcia et al.(2007)]{Garcia 2007} L. A. Garc\'{\i}a-Escudero, A.
Gordaliza, C. Matr\'{a}n, and A. Mayo-Iscar. A review of robust clustering
methods. \emph{Advances in Data Analysis and Classification},
4(2-3):89--109, 2010.

\bibitem[Kakade et al.(2009)]{Kakade 2009} S. M. Kakade, K. Sridharan, and
A. Tewari. On the complexity of linear prediction: Risk bounds, margin
bounds, and regularization. In Advances in Neural Information Processing
Systems, pp. 793--800, 2009.

\bibitem[Kakade et al.(2012)]{KakadeEtAl 2012} S. M. Kakade, S.
Shalev-Shwartz, and A. Tewari. Regularization Techniques for Learning with
Matrices. \textit{Journal of Machine Learning Research} 13:1865--1890, 2012.

\bibitem[Koltchinskii(2002)]{Koltchinskii 2002} V. Koltchinskii and D.
Panchenko, Empirical margin distributions and bounding the generalization
error of combined classifiers, \textit{The Annals of Statistics},
30(1):1--50, 2002.

\bibitem[Ledoux(1991)]{Ledoux 1991} M. Ledoux and M. Talagrand. \textit{\
Probability in Banach Spaces}, Springer, 1991.

\bibitem[Maurer and Pontil(2018)]{Maurer 2018} A. Maurer and M. Pontil.
Empirical bounds for functions with weak interactions. Proceedings of the
31st Annual Conference on Learning Theory, \emph{PMLR}, 75:987--1010, 2018.

\bibitem[Maurer(2017a)]{Maurer2017b} A. Maurer. A Second-order look at
stability and generalization. Proceedings of the 30th Annual Conference on
Learning Theory, \emph{PMLR}, 65:1461-1475, 2017.

\bibitem[Maurer(2017)]{Maurer 2017} A. Maurer. A Bernstein-type inequality
for functions of bounded interaction. \emph{Bernoulli} (Forthcoming), (see
also arXiv preprint arXiv:1701.06191).

\bibitem[McDiarmid(1998)]{McDiarmid 1998} C. McDiarmid. Concentration. In 
\textit{Probabilistic Methods of Algorithmic Discrete Mathematics}, pp.
195--248. Springer, Berlin, 1998.

\bibitem[Meir and Zhang(2003)]{Meir Zhang} R. Meir and T. Zhang.
Generalization error bounds for Bayesian mixture algorithms. \textit{Journal
of Machine Learning Research}, 4:839--860, 2003.

\bibitem[Ying et al.(2016)]{Ying} Y. Ying, L. Wen, and S. Lyu. Stochastic
online AUC maximization. In Advances in neural information processing
systems, pp. 451--459, 2016.\bigskip
\end{thebibliography}
\end{document}